\documentclass{article}
\usepackage[utf8]{inputenc}
\usepackage[T1]{fontenc}
\usepackage{amsmath}
\usepackage{amsthm}
\usepackage{amssymb}
\usepackage{amsfonts}

\usepackage[dvipsnames]{xcolor}
\usepackage{hyperref}
\usepackage[doi=false,isbn=false,url=false, giveninits=true, hyperref]{biblatex} 
\usepackage[a4paper, left=2.5cm, right=2.5cm, top=2cm, bottom=3cm, includehead, includefoot]{geometry}
\usepackage{enumitem}
\usepackage{comment}

\counterwithin{equation}{section}

\theoremstyle{plain}
\newtheorem{theorem}{Theorem}[section]

\newtheorem{lemma}[theorem]{Lemma}

\numberwithin{equation}{subsection}
\usepackage{cancel}
\theoremstyle{remark}

\theoremstyle{remark}

\theoremstyle{remark}

\theoremstyle{definition}
\newtheorem{definition}[theorem]{Definition}

\DeclareMathOperator{\supp}{supp}

\newcommand*\Diff{\mathop{}\!\mathrm{d}}

\renewcommand{\div}{\operatorname{div}}
\newcommand{\curl}{\operatorname{curl}}

\newcommand{\dist}{\operatorname{dist}}
\newcommand{\bb}[1]{\mathbb{#1}}

\newcommand\norm[1]{\left\lVert#1\right\rVert}
\newcommand*{\rom}[1]{\expandafter\@slowromancap\romannumeral #1@}

\newcommand\restr[2]{{
  \left.\kern-\nulldelimiterspace 
  #1 
  \vphantom{\big|} 
  \right|_{#2} 
  }}
  
\title{On the Vanishing Viscosity Limit for Inhomogeneous Incompressible Navier-Stokes Equations on Bounded Domains}
\author{Jens Schröder\footnote{Friedrich-Alexander-Universit\"at Erlangen-N\"urnberg, Department of Mathematics, Cauerstr.~11, 91058 Erlangen, Germany. E-mail: \href{mailto:jens.schroeder@fau.de}{jens.schroeder@fau.de}},\quad Emil Wiedemann\footnote{Friedrich-Alexander-Universit\"at Erlangen-N\"urnberg, Department of Mathematics, Cauerstr.~11, 91058 Erlangen, Germany. E-mail: \href{mailto:emil.wiedemann@fau.de}{emil.wiedemann@fau.de}}}

\counterwithout{equation}{subsection}
\counterwithout{equation}{section}
\begin{document}

\maketitle
\begin{abstract}
   In this paper we study the vanishing viscosity limit for the inhomogeneous incompressible Navier-Stokes equations on bounded domains with no-slip boundary condition in two or three space dimensions. We show that, under suitable assumptions on the density, we can establish the convergence in energy space of Leray-Hopf type solutions of the Navier-Stokes equation to a smooth solution of the Euler equations if and only if the energy dissipation vanishes on a boundary layer with thickness proportional to the viscosity. This extends Kato's criterion for homogeneous Navier-Stokes equations to the inhomogeneous case. We use a new relative energy functional in our proof.
\end{abstract}



\section{Introduction}

When setting the viscosity to zero in the Navier-Stokes equations, one formally obtains the Euler equations that describe the motion of an inviscid fluid. Whether and under which assumptions a sequence of solutions to the Navier-Stokes equations converges in a suitable sense to a solution of the Euler equations on a bounded domain is a very classical problem that was first studied by Prandtl in the beginning of the 20th century, yet many open questions remain to this day.

One of the most famous results concerning this question was given in 1984 by Kato \cite{kato}. In this paper, Kato studied the vanishing viscosity limit on bounded domains with no-slip boundary condition. More precisely, he showed that for incompressible Navier-Stokes equations with constant density, the convergence of Leray-Hopf \cite{leray} solutions to a solution of the Euler equations in energy space is equivalent to the vanishing of the energy dissipation of the viscous flow on a boundary layer with thickness proportional to the viscosity.
Kato's strategy is based on testing the weak formulation for Leray-Hopf solutions against a strong solution of the Euler equations, assuming the latter exists. Because of the mismatch between the respective boundary conditions of the Navier-Stokes (no-slip) and the Euler equations (slip), one has to introduce a suitable corrector function that cuts off the Euler solution close to the boundary. Now the difference between a Leray-Hopf solution and a strong Euler solution can be estimated in energy space by using a relative energy inequality.

Several authors later adopted variants of Kato's strategy to establish further refinements or extensions of Kato's result. Temam and Wang \cite{temam} studied the vanishing viscosity problem for infinite channel domains in $\bb{R}^2$. Wang \cite{wang} showed that instead of working with the full dissipation term $\nu\int|\nabla u^\nu|^2$, it is enough to consider $\nabla_\tau u_\tau^\nu$ or $\nabla_\tau u_n^\nu$ instead of $\nabla u^\nu$ to establish a criterion for the inviscid limit. Here $u^\nu$ denotes the viscous flow with viscosity $\nu$, and $\tau$ and $n$ denote the tangential and the normal component of the flow, respectively. This approach comes at the cost of taking a thicker boundary layer than the Kato boundary layer, whose thickness is proportional to $\nu$. 

Then Kelliher published a series of papers (\cite{kel1,kel2} and others) about the vanishing viscosity limit. His 2007 paper \cite{kel1} gives the equivalence between the vanishing viscosity limit and the vanishing of the dissipation of vorticity on the boundary layer with thickness proportional to the viscosity.

Later Sueur \cite{sueur} extended Kato's result for compressible Navier-Stokes equations by giving a criterion that ensures the convergence in the energy space of weak solutions of compressible Navier-Stokes to strong Euler solutions. Here the strategy is again based on the Kato type corrector function and the use of relative energy estimates established by Feireisl, Jin and Novotn\'y \cite{feireisl}. Recently Sueur's result was further extended to the case of density dependent viscosities \cite{caggio1,caggio2} by Bisconti, Caggio and Dell'Oro. An analogous result to the one of Wang \cite{wang} was given by Wang and Zhu for the compressible equations. In their paper \cite{zhu} they gave a criterion only considering the normal or tangential component albeit working with a thicker boundary layer.

In contrast, the vanishing viscosity limit for inhomogeneous incompressible Navier-Stokes equations has not been studied much. A first result was given by Xi \cite{xi} who extended Wang's result \cite{wang} to the inhomogeneous incompressible case and is using the same boundary layer as \cite{wang}.
In this paper, we use the strategy of the Kato type corrector function on a boundary layer of thickness proportional to $\nu$ and the use of relative energy estimates to give a criterion that characterises, in terms of the energy dissipation, the convergence of weak solutions of inhomogeneous incompressible Navier-Stokes to a strong solution of inhomogeneous incompressible Euler equations (Theorem~\ref{main_th}). In fact, we obtain precisely Kato's criterion. To our knowledge, the relative energy functional in Definition~\ref{relendef} below has not been considered before.

Let us remark that, to date, there is no rigorous example for the conceivable violation of Kato's criterion, neither in the homogeneous nor the inhomogeneous setting. In other words, it can currently not be excluded that Kato's criterion is always satisfied. However, `violent' behaviour of the Euler equations near the boundary is known, such as the breakdown of weak-strong uniqueness~\cite{BSzW, Wiedemann}. Indeed, based on physical grounds, Eyink and Quan argue in two recent preprints~\cite{EQ1, EQ2} that convergence to the smooth Euler solution should generally \emph{not} be expected due to boundary layer formation; in such cases, in particular, Kato's criterion would not be satisfied.

\section{The Model}
\subsection{The Inhomogeneous Incompressible Navier-Stokes Equations}
Let $\Omega\subset \bb{R}^n$ be a bounded, open, connected and simply connected subset with $C^2$-boundary. Let $u_\nu:[0,\infty)\times\Omega\to \bb{R}^n$ and $\rho_\nu:[0,\infty)\times\Omega\to \bb{R}$, $\nu>0$, $p_\nu:[0,\infty)\times\Omega\to \bb{R}$ be functions that depend on the viscosity parameter $\nu>0$. Here the physical meaning of $u_\nu$ is the velocity of the fluid, $p_\nu$ its pressure and $\rho_\nu$ is the density of the fluid.

For any $\nu>0$, we consider the inhomogeneous incompressible Navier-Stokes equations (NSE):

    \begin{align}
        \label{nse1} \partial_t (\rho_\nu u_\nu) +\div(\rho_\nu  u\otimes u_\nu) +\nabla p_\nu-\nu \Delta u_\nu&=0\\
    \label{nse2} \partial_t \rho_\nu+\div(\rho_\nu u_\nu)&=0 \\ 
    \label{nse3} \div(u_\nu)&=0,
       \end{align}
    together with the Dirichlet boundary condition $\restr{u_\nu}{\partial \Omega}=0$ and the initial conditions $\rho_\nu(0)=\rho_0$ and $u_\nu(0)=u_0$ such that $\div (u_0)=0$ and $\restr{u_0}{\partial\Omega}=0$ 
    for any $\nu>0$. For simplicity, the initial conditions are taken independent of viscosity.

For the weak formulation, we write as usual $H_0^1(\Omega)$ for the functions with square integrable weak derivatives and zero boundary trace, and $H(\Omega)$ for the solenoidal square integrable vector fields (i.e., those that are divergence-free and have zero normal trace on the boundary in a suitable sense).
		
We define Leray-Hopf type solutions of the system of equations (\ref{nse1})-(\ref{nse3}) in the following sense:\\
Let $\rho_\nu\in L^\infty((0,\infty)\times\Omega)$ and let $u_\nu\in L^2(0,\infty;(H\cap H^1_0)(\Omega))$.
Then $\rho_\nu$ and $u_\nu$ are considered a weak solution if they satisfy the weak momentum equation 
\begin{equation}
\label{moment_1}
-\int_\Omega \rho_0 u_0 \Phi(0,x)\Diff x=\int_0^\infty\int_\Omega \rho_\nu u_\nu  \partial_t\Phi+\rho_\nu u_\nu\otimes u_\nu :\nabla \Phi- \nu \nabla u_\nu:\nabla\Phi \Diff x\Diff t
\end{equation}
for all $\Phi\in C_c^\infty([0,\infty)\times \Omega,\bb{R}^n)$ with $\div\Phi(t,x)=0$,
the weak transport equation
\begin{equation}
\label{trans_1}
-\int_\Omega \rho_0 \theta(0,x)\Diff x=\int_0^\infty\int_\Omega \rho_\nu \partial_t\theta+\rho_\nu u_\nu \cdot\nabla\theta\Diff x\Diff t
\end{equation}
for all $\theta\in C_c^\infty([0,\infty)\times \overline{\Omega},\bb{R})$, and
the energy inequality
 \begin{equation}
 \label{ineq}
 \frac{1}{2}\int_\Omega\rho_\nu |u_\nu|^2\Diff x+\nu \int_0^T\int_\Omega|\nabla u_\nu|^2 \Diff x\Diff t\leq \frac{1}{2}\int_\Omega\rho_0 |u_0|^2\Diff x
  \end{equation}
  for a.e. $T\in(0,\infty)$.
  The global existence of weak solutions to the inhomogeneous incompressible Navier-Stokes equation in the sense specified above was first established by Kazhikov \cite{kazi} and further results have been published by Kazhikov and Smagulov \cite{kazi2}, Lions \cite{Lions} and Simon \cite{simon}.
  \begin{theorem}(\cite{Lions})
  Consider the initial data $\rho_0\in L^\infty((0,\infty)\times \Omega)$ and $u_0\in L^2(0,\infty;H(\Omega))$. Then there exists a solution $(u_\nu,\rho_\nu)$ satisfying $(\ref{moment_1})$, $(\ref{trans_1})$ and $(\ref{ineq})$
   and lying in the spaces 
       $\rho_\nu\in L^\infty([0,\infty)\times \Omega)$, $u_\nu\in L^2(0,T,(H\cap H_0^1)(\Omega))$,  $\rho_\nu \in C([0,T),L^p(\Omega))$ for all $1\leq p<\infty$ and all $T\in(0,\infty)$. Moreover,
			\begin{equation*}
			\inf_{x\in\Omega}\rho_0(x)\leq \rho(x,t)\leq\sup_{x\in\Omega}\rho_0(x)<\infty
        			\end{equation*}
							for almost all $(t,x)\in(0,\infty)\times{\Omega}$.
  \end{theorem}
  Note that we can restrict equations (\ref{moment_1}) and (\ref{trans_1}) to any $0<T<\infty$, which leads to following equations:
  \begin{equation}
\label{moment}
\int_\Omega \rho_\nu(T) u_\nu(T) \Phi(T)\Diff x-\int_\Omega \rho_0 u_0 \Phi(0,x)\Diff x=\int_0^T\int_\Omega \rho_\nu u_\nu  \partial_t\Phi+\rho_\nu u_\nu\otimes u_\nu :\nabla \Phi- \nu \nabla u_\nu:\nabla\Phi \Diff x\Diff t,
\end{equation}
and
\begin{equation}
\label{trans}
\int_\Omega \rho_\nu(T) \theta(T)\Diff x-\int_\Omega \rho_0 \theta(0,x)\Diff x=\int_0^T\int_\Omega \rho_\nu \partial_t\theta+\rho_\nu u_\nu \cdot\nabla\theta\Diff x\Diff t.
\end{equation}

\subsection{The Inhomogeneous Incompressible Euler Equations}
Setting $\nu=0$ in equation (\ref{nse1}), the second order derivatives vanish and one obtains formally the inhomogeneous incompressible Euler equations (EE):
    \begin{align}
        \label{ee1} \partial_t (\rho u) +\div(\rho  u\otimes u) +\nabla p=0\\
    \label{ee2} \partial_t \rho+\div(\rho u)=0 \\ 
    \label{ee3} \div(u)=0,
       \end{align}
       where $u:[0,\infty)\times\Omega\to \bb{R}^n$ and $\rho:[0,\infty)\times\Omega\to \bb{R}$, $p:[0,\infty)\times\Omega\to \bb{R}$.
			
    As for the Navier-Stokes equations, we impose the initial conditions  $\rho(0,x)=\rho_0(x)$ and $u(0,x)=u_0(x)$ satisfying  $\restr{u_0\cdot n}{\partial\Omega}=0$, where $n$ is the outward pointing normal to $\partial\Omega$. In contrast to the Navier-Stokes equations, we work with a first order equation here, hence we only impose the slip condition $\restr{u\cdot n}{\partial \Omega}=0$.
		
    Here we are interested in classical solutions to (\ref{ee1})-(\ref{ee3}), hence due to the regularity of the functions we can rewrite the equation in the following way:
    \begin{align}
        \label{pee1} \rho \partial_t u +\rho(u\cdot \nabla)u +\nabla p&=0\\
    \label{pee2} \partial_t \rho+u\cdot \nabla \rho&=0 \\ 
    \label{pee3} \div(u)&=0.
       \end{align}
       Local existence of classical solutions to the system (\ref{pee1})-(\ref{pee3}) was established by Valli and Zaj\k aczkowski with the following assumptions:
    \begin{theorem}(\cite{valli})
    \label{thm_valli}
        Let $\Omega\subset\bb{R}^n$ be a bounded open domain with $\partial\Omega\in C^4$. Let $p>n$. Let $u_0\in W^{2,p}(\Omega)$ with $\div(u_0)=0$, $\restr{u_0\cdot n}{\partial\Omega}=0$. Let $\rho_0\in W^{2,p}(\Omega)$ such that $0<\inf_{x\in\Omega}\rho_0(x)\leq \rho(x)\leq \sup_{x\in\Omega}\rho_0(x)<\infty$. Then there exists $t_0>0$ and functions $\rho\in C^0([0,t_0];W^{2,p}(\Omega))\cap C^1([0,t_0];W^{1,p}(\Omega))$ and $u\in C^0([0,t_0];W^{2,p}(\Omega))\cap C^1([0,t_0];W^{1,p}(\Omega))$ and $p\in C^0([0,t_0];W^{3,p}(\Omega))$ such that $(\rho,u,p)$ satisfy the inhomogeneous incompressible Euler equations (\ref{ee1})-(\ref{ee3}).
				
        Furthermore 
        $$0<\inf_{x\in\Omega}\rho_0(x)\leq \rho(x,t)\leq\sup_{x\in\Omega}\rho_0(x)<\infty$$
        for all $(t,x)\in[0,t_0]\times\overline{\Omega}$.
    \end{theorem}
 Note that Theorem \ref{thm_valli} implies the existence of classical solutions to the Euler equations locally in time. This is due to the fact that from the Sobolev embedding we have $W^{2,p}(\Omega)\subset C^{1,1-\frac{n}{p}}(\overline{\Omega})$ whenever $p>n$. In particular the solution satisfies the energy equality
 \begin{equation}
 \label{ineq_ee}
 \int_\Omega\rho(t) |u(t)|^2\Diff x= \int_\Omega\rho_0 |u_0|^2\Diff x
  \end{equation}
	for every $t\in[0,t_0]$.

\section{The Main Result}
The main result of this paper is the following theorem that gives a sufficient condition on the energy dissipation in order to have convergence of a sequence of weak solutions of the Navier-Stokes equation to a strong solution of the Euler equation.
 \begin{theorem}\label{main_th}
Let $T>0$ and $(\rho,u)$ be a classical ($C^1$) solution of (EE)  on $[0,T)\times\Omega$ with the initial condition $u(0,x)=u_0(x)\in L^2(\Omega)$ and $\rho(0,x)=\rho_0(x)\in L^\infty(\Omega)$ and the boundary condition $\restr{u\cdot n}{\partial\Omega}=0$.

We pick a sequence of viscosities $(\nu_n)_{n\in\bb{N}}>0$ such that $\nu_n\to 0$. (In the following we drop the index $n$ to simplify the notation). Let $(\rho_\nu,u_\nu)$ be weak solutions of (NSE) on $[0,T]\times\Omega$, in the sense that they satisfy (\ref{moment_1}), (\ref{trans_1}) and the energy inequality (\ref{ineq}) for the initial condition $u_\nu(0)=u_0\in L^2(\Omega)$ and $\rho_\nu(0)=\rho_0\in L^\infty(\Omega)$ and the Dirichlet boundary condition $\restr{u_\nu}{\partial\Omega}=0$. 

Further we suppose a uniform bound from below and above on $\rho_0$:
 There exist $c,C>0$ such that $c<\rho_0(x)<C$ for a.e $x\in \Omega$. Set $\Omega
_\nu=\{x\in\Omega:\dist(x,\partial\Omega)\leq \nu\}$. Finally, let $0<T'<T$. 

Then $\norm{u_\nu-u}_{L_{loc}^\infty([0,T);L^2(\Omega))}\to 0$ and $\norm{\rho_\nu-\rho}_{L_{loc}^\infty([0,T');L^2(\Omega))}\to 0$, as $\nu\to 0$, if and only if

\begin{equation}\label{katocrit}
\nu\int_0^{T'}\int_{\Omega_\nu}|\nabla u_\nu|^2 \Diff x\Diff t\to 0 \quad \text{as}\quad \nu\to 0.
\end{equation}

\end{theorem}
The `if' part is more difficult. In order to prove it, we use the relative energy functional that measures the distance between $(\rho_\nu,u_\nu)$ and $(\rho,u)$. The relative energy method has frequently been applied for the study of the weak-strong uniqueness of Navier-Stokes equations and related models, see~\cite{Wiedemann}. In a recent work by Crin-Barat, \v{S}kondri\'{c} and Violini, the authors applied the relative energy functional to establish the weak-strong uniqueness of inhomogeneous Navier–Stokes equations far from vacuum \cite{stefan}.

 \begin{definition}\label{relendef}
 We define the relative energy functional
    \begin{equation}\label{relenergydef} 
		E^\nu(t)=E_1^\nu(t)+E_2^\nu(t)=\frac{1}{2}\int_\Omega \rho_\nu(t) |u_\nu(t)-u(t)|^2\Diff x+\int_\Omega|\rho_\nu(t)-\rho(t)|^2\Diff x,
\end{equation}
     where $(\rho_\nu,u_\nu) $ is a solution of (NSE) and $(\rho, u)$ is a solution of (EE).
 \end{definition}

 Note that if we consider $\rho_\nu=\rho=\text{const}$ in Theorem \ref{main_th}, the criterion reduces to the criterion given by Kato \cite{kato} for the homogeneous incompressible equations. To our knowledge, the relative energy~\eqref{relenergydef} has not been used before\footnote{We thank Stefan \v{S}kondri\'{c} and Alessandro Violini for stimulating discussions on the choice of relative energy.}.

\section{The Proof of the Main Result}
\begin{proof}
{\em Step 1: Kato's criterion implies convergence in the viscosity limit.} Assume~\eqref{katocrit}. The strategy of the proof of convergence here is similar to the one used by Kato for the case where the density is constant. First, we establish relative energy estimates to then introduce a corrector function which allows us to make a special choice for the test functions.
    \begin{align*}
     &E_1^\nu(s)=\frac{1}{2}\int_\Omega \rho_\nu(s) |u_\nu(s)-u(s)|^2\Diff x\\
     &=\frac{1}{2}\int_\Omega \rho_\nu(s) |u_\nu(s)|^2\Diff x+\frac{1}{2}\int_\Omega \rho_\nu(s) |u(s)|^2\Diff x-\int_\Omega \rho_\nu(s) u_\nu(s)\cdot u(s)\Diff x\\
     &=\frac{1}{2}\int_\Omega \rho_\nu(s) |u_\nu(s)|^2\Diff x+\frac{1}{2}\int_\Omega \rho(s) |u(s)|^2\Diff x\\
     &-\int_\Omega \rho_\nu(s) u_\nu(s)\cdot u(s)\Diff x+\frac{1}{2}\int_\Omega (\rho_\nu(s)-\rho(s)) u(s)\cdot u(s)\Diff x\\
     &\leq \int_\Omega \rho_0 |u_0|^2\Diff x-\nu\int_0^s \int_\Omega|\nabla u_\nu|^2\Diff x\Diff t-\int_\Omega \rho_\nu(s) u_\nu(s)\cdot u(s)\Diff x+\frac{1}{2}\int_\Omega (\rho_\nu(s)-\rho(s)) u(s)\cdot u(s)\Diff x,\\
 \end{align*}
 where we used the energy (in)equality in the last step.
 
 Before continuing the estimate, we establish some identities:
 First we want to use the weak formulation of the (NSE).
 For this use $u-\Phi_\nu$ as a test function, where $u$ is the smooth solution of the (EE) and $\Phi_\nu$ is chosen such that $\Phi_\nu\in C^\infty(\Omega)$, $\supp(\Phi_\nu)=\{x\in\Omega:\dist(x,\partial\Omega)\leq \nu\}=:\Omega_\nu$, $\restr{(u-\Phi_\nu)}{\partial\Omega}=0$, $\div(u-\Phi_\nu)=0$ on $[0,\infty)\times \Omega$. 
 Furthermore $\Phi_\nu$ satisfies (see Appendix \ref{appendix_1})
 \begin{enumerate}
  \item $\norm{\Phi_\nu}_{L^\infty(\Omega)}\leq K$
     \item $\norm{\nabla\Phi_\nu}_{L^\infty(\Omega)}\lesssim \nu^{-1}$
     \item $\norm{\Phi_\nu}_{L^2(\Omega)}\lesssim \nu^\frac{1}{2}$
     \item $\norm{\partial_t\Phi_\nu}_{L^2(\Omega)}\lesssim \nu^\frac{1}{2}$
     \item $\norm{\nabla\Phi_\nu}_{L^2(\Omega)}\lesssim \nu^{-\frac{1}{2}}$
 \end{enumerate}
 Hence equation $(\ref{moment})$ becomes
 \begin{equation*}
 \begin{split}
&\int_\Omega \rho_\nu(s) u_\nu(s)\cdot (u(s)-\Phi_\nu(s))\Diff x-\int_\Omega \rho_0 u_0\cdot(u_0-\Phi_\nu(0)) \Diff x\\
&=\int_0^s\int_\Omega \rho_\nu u_\nu\cdot  \partial_t(u-\Phi_\nu)+\rho_\nu u_\nu\otimes u_\nu :\nabla (u-\Phi_\nu)- \nu \nabla u_\nu:\nabla(u-\Phi_\nu) \Diff x\Diff t.
 \end{split}
\end{equation*}
Using that $\norm{\Phi_\nu}_{L^2(\Omega)}\lesssim \nu^\frac{1}{2},$ $\norm{\partial_t\Phi_\nu}_{L^2(\Omega)}\lesssim \nu^\frac{1}{2} $ and that $\norm{\rho_\nu}_{L^\infty(\Omega)}\leq C$, we get that in the limit $\nu\to 0$

\begin{equation}
\label{testphi}
 \begin{split}
&o(1)+\int_\Omega \rho_\nu(s) u_\nu(s) \cdot u(s)\Diff x-\int_\Omega \rho_0 u_0\cdot u_0 \Diff x\\
&=\int_0^s\int_\Omega \rho_\nu u_\nu  \cdot\partial_t u+\rho_\nu u_\nu\otimes u_\nu :\nabla (u-\Phi_\nu)- \nu \nabla u_\nu:\nabla(u-\Phi_\nu) \Diff x\Diff t.
 \end{split}
\end{equation}
Further note the identity
\begin{equation}
\label{id1}
    \rho_\nu u_\nu\cdot\partial_t u=\rho u_\nu\cdot \partial_t u+(\rho_\nu-\rho)u\cdot \partial_t u+(\rho_\nu-\rho)(u_\nu-u)\cdot \partial_t u.
\end{equation}
Using the Euler equation (\ref{pee1}), multiplying it with $u_\nu$ and integrating we get
\begin{equation}
\label{id2}
    \begin{split}
        0&=\int_0^s\int_\Omega \rho \partial_t u\cdot u_\nu +\rho(u\cdot \nabla)u\cdot u_\nu +\nabla p\cdot u_\nu\Diff x\Diff t\\
    &=\int_0^s\int_\Omega \rho \partial_t u\cdot u_\nu +\rho(u\cdot \nabla)u\cdot u_\nu \Diff x\Diff t,
    \end{split}
\end{equation}
    
where we used the weak incompressibility in the second equality. Hence, using (\ref{id1}) and (\ref{id2}),
\begin{equation}
\label{id3}
    \begin{split}
         \int_0^s\int_\Omega \rho_\nu u_\nu \cdot\partial_t u\Diff x\Diff t&=\int_0^s\int_\Omega \rho \partial_t u\cdot u_\nu+(\rho_\nu-\rho) \partial_t u\cdot u+(\rho_\nu-\rho) \partial_t u\cdot (u_\nu-u)\Diff x\Diff t\\
    &=\int_0^s\int_\Omega -\rho(u\cdot\nabla)u\cdot u_\nu+(\rho_\nu-\rho) \partial_t u\cdot u+(\rho_\nu-\rho) \partial_t u\cdot (u_\nu-u)\Diff x\Diff t.\\
    \end{split}
\end{equation}

Now we use $|u|^2$ as a test function in the transport equation (\ref{trans}):
\begin{align*}
    \int_\Omega \rho_\nu u\cdot u\Diff x-\int_\Omega \rho_0 u_0\cdot u_0\Diff x=\int_0^s\int_\Omega \rho_\nu \partial_t(u\cdot u)+\rho_\nu u_\nu \cdot\nabla(u\cdot u)\Diff x\Diff t,
\end{align*}
which due to the smoothness of $u$ becomes
\begin{align*}
    \int_\Omega \rho_\nu u\cdot u\Diff x-\int_\Omega \rho_0 u_0\cdot u_0\Diff x=2\int_0^s\int_\Omega \rho_\nu u\cdot \partial_t u+\rho_\nu (u_\nu\cdot \nabla)u\cdot u\Diff x\Diff t.
\end{align*}
Similarly, for the transport equation of the Euler system we get
\begin{align*}
    \int_\Omega \rho u\cdot u\Diff x-\int_\Omega \rho_0 u_0\cdot u_0\Diff x=2\int_0^s\int_\Omega \rho u\cdot \partial_t u+\rho (u\cdot \nabla)u\cdot u\Diff x\Diff t.
\end{align*}
Subtracting both identities from each other gives
\begin{equation}
\label{eq3}
 \int_\Omega (\rho_\nu-\rho) u\cdot u\Diff x=2\int_0^s\int_\Omega (\rho_\nu-\rho) u\cdot \partial_t u+(\rho_\nu (u_\nu\cdot \nabla)-\rho (u\cdot \nabla))u\cdot u\Diff x\Diff t.
\end{equation}
   
Hence we continue the estimate of $E_1^\nu(s)$:
\begin{align*}
    &E_1^\nu(s)\\
    &\leq \int_0^s \int_\Omega-\rho_\nu u_\nu  \cdot\partial_t u-\rho_\nu u_\nu\otimes u_\nu :\nabla (u-\Phi_\nu)+ \nu \nabla u_\nu:\nabla(u-\Phi_\nu)-\nu|\nabla u_\nu|^2\Diff x\Diff t\\
    &+\frac{1}{2}\int_\Omega (\rho_\nu-\rho) u\cdot u\Diff x+o(1)\\
    &=\int_0^s \int_\Omega \rho(u\cdot\nabla)u\cdot u_\nu-(\rho_\nu-\rho)u\cdot \partial_t u-(\rho_\nu-\rho)(u_\nu-u)\cdot \partial_t u\\
    &-\rho_\nu (u_\nu\cdot\nabla)u\cdot u_\nu+\rho_\nu(u_\nu\cdot\nabla)\Phi_\nu\cdot u_\nu+ \nu \nabla u_\nu:\nabla(u-\Phi_\nu)-\nu|\nabla u_\nu|^2\Diff x\Diff t\\
    &+\frac{1}{2}\int_\Omega (\rho_\nu-\rho) u\cdot u\Diff x+o(1)\\
    &=\int_0^s \int_\Omega \rho(u\cdot\nabla)u\cdot u_\nu-\cancel{(\rho_\nu-\rho)u\cdot \partial_t u}-(\rho_\nu-\rho)(u_\nu-u)\cdot \partial_t u\\
    &-\rho_\nu (u_\nu\cdot\nabla)u\cdot u_\nu+\rho_\nu(u_\nu\cdot\nabla)\Phi_\nu\cdot u_\nu+ \nu \nabla u_\nu:\nabla(u-\Phi_\nu)-\nu|\nabla u_\nu|^2\\
    &+\cancel{(\rho_\nu-\rho) u\cdot \partial_t u}+(\rho_\nu (u_\nu\cdot \nabla)-\rho (u\cdot \nabla))u\cdot u\Diff x\Diff t+o(1)\\
    &=\int_0^s \int_\Omega \rho(u\cdot\nabla)u\cdot u_\nu-(\rho_\nu-\rho)(u_\nu-u)\cdot \partial_t u\\
    &-\rho_\nu (u_\nu\cdot\nabla)u\cdot u_\nu+\rho_\nu(u_\nu\cdot\nabla)\Phi_\nu\cdot u_\nu+ \nu \nabla u_\nu:\nabla(u-\Phi_\nu)-\cancel{\nu|\nabla u_\nu|^2}\\
    &+(\rho_\nu (u_\nu\cdot \nabla)-\rho (u\cdot \nabla))u\cdot u\Diff x\Diff t+o(1)\\
    &\leq \int_0^s \int_\Omega \rho(u\cdot\nabla)u\cdot u_\nu-\rho_\nu (u_\nu\cdot\nabla)u\cdot u_\nu+\rho_\nu (u_\nu\cdot \nabla)u\cdot u-\rho (u\cdot \nabla)u\cdot u\\
    &+\rho_\nu(u_\nu\cdot\nabla)\Phi_\nu\cdot u_\nu-(\rho_\nu-\rho)(u_\nu-u)\cdot \partial_t u+ \nu \nabla u_\nu:\nabla(u-\Phi_\nu)\Diff x\Diff t\\
    &+o(1)\\
    &\leq \int_0^s \int_\Omega \underbrace{-\rho_\nu((u_\nu-u)\cdot\nabla)u\cdot(u_\nu-u)}_{I_1}+\underbrace{(\rho_\nu-\rho)(u\cdot \nabla)u\cdot(u-u_\nu)}_{I_2}\\
    &+\underbrace{\rho_\nu(u_\nu\cdot\nabla)\Phi_\nu\cdot u_\nu}_{I_3}-\underbrace{(\rho_\nu-\rho)(u_\nu-u)\cdot \partial_t u}_{I_4}+ \underbrace{\nu \nabla u_\nu:\nabla(u-\Phi_\nu)}_{I_5}\Diff x\Diff t\\
    &+o(1),
\end{align*}

where we used (\ref{testphi}) for the first step, (\ref{id3}) for the second step, (\ref{eq3}) in the third step. In the 4th, 5th, and 6th steps we rearranged and simplified the expression. Now we can estimate the above expressions. 

For $I_5$ we have
\begin{align*}
    \left|\nu\int_0^s\int_\Omega  \nabla u_\nu:\nabla(u-\Phi_\nu)\Diff x\Diff t\right|&\leq \nu\left(\int_0^s\norm{\nabla u_\nu}_{L^2(\Omega)}\norm{\nabla u}_{L^2(\Omega)}+\norm{\nabla u_\nu}_{L^2(\Omega_\nu)}\norm{\nabla \Phi_\nu}_{L^2(\Omega_\nu)}\Diff t\right)\\
    &\leq C\nu \int_0^s\norm{\nabla u_\nu}_{L^2(\Omega)}\Diff t+C\nu^\frac{1}{2}\int_0^s\norm{\nabla u_\nu}_{L^2(\Omega_\nu)}\Diff t\\
    &\leq Cs^\frac{1}{2}\nu^\frac{1}{2}\left(\underbrace{\nu\int_0^s\int_{\Omega}|\nabla u_\nu|^2 \Diff x\Diff t}_{\leq C}\right)^\frac{1}{2}+Cs^\frac{1}{2}\left(\underbrace{\nu\int_0^s\int_{\Omega_\nu}|\nabla u_\nu|^2 \Diff x\Diff t}_{\to 0\quad\text{as}\quad \nu\to 0}\right)^\frac{1}{2}\\
    &\longrightarrow 0 \quad\text{as}\quad \nu\to 0.
\end{align*}

For $I_3$ we compute

\begin{align*}
    \left|\int_0^s\int_\Omega\rho_\nu(u_\nu\cdot\nabla)\Phi_\nu\cdot u_\nu\Diff x\Diff t\right|&\leq \int_0^s\int_\Omega\rho_\nu|u_\nu||\nabla\Phi_\nu|| u_\nu|\Diff x\Diff t\\
    &=\int_0^s\int_{\Omega_\nu}\rho_\nu|u_\nu||\nabla\Phi_\nu|| u_\nu|\Diff x\Diff t\\
    &\leq C\int_0^s\underbrace{\norm{\dist(x,\partial\Omega)^2|\nabla\Phi_\nu|}_{L^\infty}}_{\leq C\nu}\int_{\Omega_\nu}\frac{1}{\dist(x,\partial\Omega)^2}|u_\nu|^2\Diff x\Diff t\\
    &\leq C\nu\int_0^s\int_{\Omega_\nu}|\nabla u_\nu|^2\Diff x\Diff t\\
    &\longrightarrow 0,
\end{align*}
where we used $\rho_\nu\leq C$ a.e., the Hölder inequality in the third step and the Hardy inequality (see Appendix \ref{appendix_2}) for the last step. Note that $\Phi_\nu$ takes its support in $\Omega_\nu$ thus the integral can be restricted to $\Omega_\nu$.

For $I_1$: 
\begin{align*}
    \left|\int_0^s \int_\Omega\rho_\nu((u_\nu-u)\cdot\nabla)u\cdot(u_\nu-u)\Diff x\Diff t\right|&\leq \norm{\nabla u}_{L^\infty}\int_0^s \int_\Omega\rho_\nu|u_\nu-u|^2\Diff x\Diff t\\
    &\leq C\int_0^s E_1^\nu(s)\Diff t.
\end{align*}

    For $I_2$:
    \begin{align*}
       \left|\int_0^s\int_\Omega (\rho_\nu-\rho)(u\cdot \nabla)u\cdot(u-u_\nu)\Diff x\Diff t\right|&\leq \norm{\nabla u}_{L^\infty}\norm{\rho_\nu-\rho}_{L^2([0,s]\times\Omega)}\norm{u_\nu-u}_{L^2([0,s]\times\Omega)}\\
       &\leq \frac{C}{2} \norm{\rho_\nu-\rho}^2_{L^2([0,s]\times\Omega)}+\frac{1}{2}\norm{u_\nu-u}^2_{L^2([0,s]\times\Omega)}\\
       &\leq C\int_0^s E_1^\nu(s)+E_2^\nu(s)\Diff t,
    \end{align*}
    where we used Hölder's and Young's inequality and the fact that $\rho_\nu$ is uniformly bounded from below.

    For $I_4$, we estimate similarly as for $I_2$:
    \begin{align*}
        \left|\int_0^s\int_\Omega (\rho_\nu-\rho)(u_\nu-u)\cdot \partial_t u\Diff x\Diff t\right|&\leq \norm{\partial_t u}_{L^\infty([0,s]\times\Omega)}\norm{\rho_\nu-\rho}_{L^2([0,s]\times\Omega)}\norm{u_\nu-u}_{L^2([0,s]\times\Omega)}\\
       &\leq \frac{C}{2} \norm{\rho_\nu-\rho}^2_{L^2([0,s]\times\Omega)}+\frac{1}{2}\norm{u_\nu-u}^2_{L^2([0,s]\times\Omega)}\\
       &\leq C\int_0^s E_1^\nu(s)+E_2^\nu(s)\Diff t.
    \end{align*}
Finally for $E_2^\nu$:
    \begin{align*}
        E_2^\nu(s)&=\frac{1}{2}\int_\Omega|\rho_\nu(s)-\rho(s)|^2\Diff x\\
        &=\frac{1}{2}\norm{\rho_\nu}_{L^2(\Omega)}^2+\frac{1}{2}\norm{\rho}_{L^2(\Omega)}^2-\int_\Omega \rho_\nu(s)\rho(s)\Diff x\\
        &=\norm{\rho_0}_{L^2(\Omega)}^2-\int_\Omega \rho_\nu(s)\rho(s)\Diff x.\\
    \end{align*}
    The weak formulation of the transport equation with $\rho$ as a test function gives
    $$\int_\Omega \rho_\nu \rho\Diff x-\int_\Omega \rho_0 \rho_0\Diff x=\int_0^s\int_\Omega \rho_\nu \partial_t\rho+\rho_\nu u_\nu \cdot\nabla\rho\Diff x\Diff t.$$
    Hence
    \begin{align*}
        E_2^\nu(s)&\leq-\int_0^s\int_\Omega \rho_\nu \partial_t\rho+\rho_\nu u_\nu \cdot\nabla\rho\Diff x\Diff t\\
        &=\int_0^s\int_\Omega \rho_\nu(u-u_\nu)\cdot\nabla\rho\Diff x\Diff t-\underbrace{\int_0^s\int_\Omega \rho(u-u_\nu)\cdot\nabla\rho\Diff x\Diff t}_{=0}\\
        &=\int_0^s\int_\Omega (\rho_\nu-\rho)(u-u_\nu)\cdot\nabla\rho\Diff x\Diff t\\
        &\leq C\norm{\rho-\rho_\nu}_{L^2([0,s]\times\Omega)}^2+C\norm{u-u_\nu}_{L^2([0,s]\times\Omega)}^2\\
        &\leq C\int_0^s\left(\norm{\rho-\rho_\nu}_{L^2(\Omega)}^2+\norm{\sqrt{\rho_\nu}(u-u_\nu)}_{L^2(\Omega)}^2\right)\Diff t\\
        &\leq C\int_0^s E_1^\nu(t)+E_2^\nu(t)\Diff t.
    \end{align*}

    Hence, we have shown that 
    $$E_1^\nu(s)+E_2^\nu(s)\leq C\int_0^s E_1^\nu(t)+E_2^\nu(t)+o(1)\Diff t$$
    for some constant $C>0$ that does not depend on $\nu$.
    Thus by Grönwall's lemma we conclude that $E^\nu(s)=E_1^\nu(s)+E_2^\nu(s)\rightarrow 0$ in the limit $\nu\to 0$.
		
	{\em Step 2: Convergence in the viscosity limit implies Kato's criterion.}	For the converse assertion, suppose $\rho_{\nu}\to\rho$ and $u_n\to u$ both in $L^\infty(0,T;L^2(\Omega))$.
		 By the energy inequality (\ref{ineq}) we have
    \begin{align}\label{disscontrol}
        \limsup_{\nu\to 0}\nu \int_0^t\int_\Omega|\nabla u_\nu|^2 \Diff x\Diff s&\leq \limsup_{\nu\to 0} \left(\frac{1}{2}\int_\Omega\rho_0 |u_0|^2\Diff x-\frac{1}{2}\int_\Omega\rho_\nu(t) |u_\nu(t)|^2\Diff x\right)
    \end{align}
   for almost every $t\in(0,T)$. Pick a time $T'<\tau< T$ for which~\eqref{disscontrol} is valid, and for which $\rho_\nu(\tau)\to\rho(\tau)$ and $u_\nu(\tau)\to u(\tau)$ in $L^2(\Omega)$, respectively.  We show the convergence to zero of the right hand side of~\eqref{disscontrol} (with $t$ replaced by $\tau$) with the use of the Vitali convergence theorem. First note that due to the $L^2$ convergence in $\Omega$ we can extract subsequences (not relabeled) such that $\rho_\nu(\tau)\to \rho(\tau)$ and $u_\nu(\tau)\to u(\tau)$ almost everywhere in $\Omega$. Hence, along a further subsequence we have  $\rho_\nu|u_\nu|^2\to \rho|u|^2$ almost everywhere, evaluated at $\tau$. Since $\Omega$ is bounded, the convergence in measure of $\rho_\nu|u_\nu|^2\to \rho|u|^2$, evaluated at $\tau$, follows.
		
		As $u_\nu(\tau)\to u(\tau)$ in $L^2(\Omega))$, in particular $|u_\nu(\tau)|^2$ converges in $L^1(\Omega)$, so that the Dunford-Pettis Theorem yields equi-integrability of $|u_\nu(\tau)|^2$, and hence of $\rho_\nu(\tau) |u_\nu(\tau)|^2$ as $(\rho_\nu)$ is uniformly bounded. 
		
    Hence the Vitali convergence theorem applies, i.e.,
     $$\lim_{\nu\to 0}\int_\Omega \rho_\nu(\tau)|u_\nu(\tau)|^2\Diff x=\int_\Omega \rho(\tau)|u(\tau)|^2\Diff x=\int_\Omega \rho_0|u_0|^2\Diff x$$
     by virtue of the energy conservation for the smooth solution of Euler. Together with~\eqref{disscontrol} this yields
		\begin{equation*}
		\lim_{\nu\to 0}\nu \int_0^{\tau}\int_\Omega|\nabla u_\nu|^2 \Diff x\Diff s=0
		\end{equation*}
    for each $k$, and, since $\tau>T'$, 
		\begin{equation*}
		\lim_{\nu\to 0}\nu \int_0^{T'}\int_\Omega|\nabla u_\nu|^2 \Diff x\Diff s=0
		\end{equation*}
		as claimed.
\end{proof}


  

\section{Appendix}
\subsection{Estimates on the Corrector Function \texorpdfstring{$\Phi_\nu$}{Lg}}
\label{appendix_1}
 First, we give a method to explicitly construct the corrector function $\Phi_\nu$ on a bounded, open, connected and simply connected subset $\Omega\subset \bb{R}^3$.

  Let $\eta\in C_c^\infty([0,\infty))$ be such that $\eta(0)=1,\eta(x)=0$ for all $x\geq 1$ and $\eta$ is monotone decreasing. Now define $\eta_\nu(x):=\eta\left(\frac{\dist(x,\partial\Omega)}{\nu}\right)$ for all $x\in\Omega.$ Thus $\eta_\nu\in C_c^\infty(\Omega)$ satisfies 
 $\eta_\nu(0)=1$ and $\supp(\eta_\nu)\subset\Omega_\nu$.

 Since $\div(u)=0$, by the Poincaré Lemma there exists a smooth vector field $\Phi:\Omega\to\bb{R}^3$ such that $\curl(\Phi)=u$ and $\Phi=0$ on $\partial\Omega$.

 Now define a new vector field $\Phi_\nu:\Omega\to\bb{R}^3$ by $\Phi_\nu=\curl(\eta_\nu \Phi)=\eta_\nu\curl(\Phi)+\nabla\eta_\nu\times\Phi$. Then $\Phi_\nu\in C^\infty(\Omega)$ with $\supp(\Phi_\nu)\subset\Omega_\nu$. Furthermore we have $\div(\Phi_\nu)\equiv 0$ on $\Omega$ and $\restr{\Phi_\nu}{\partial\Omega}=u$.

 Now we check the bounds
  \begin{enumerate}[label=(\roman*)]
  \item $\norm{\Phi_\nu}_{L^\infty(\Omega)}\leq K$
     \item $\norm{\nabla\Phi_\nu}_{L^\infty(\Omega)}\lesssim \nu^{-1}$
     \item $\norm{\Phi_\nu}_{L^2(\Omega)}\lesssim \nu^\frac{1}{2}$
     \item $\norm{\partial_t\Phi_\nu}_{L^2(\Omega)}\lesssim \nu^\frac{1}{2}$
     \item $\norm{\nabla\Phi_\nu}_{L^2(\Omega)}\lesssim \nu^{-\frac{1}{2}}$.
 \end{enumerate}
 As a general remark first note that $|\Omega_\nu|\leq C\nu$ and that $|\Phi(x,t)|\leq C\nu$ on $\Omega_\nu$ since $\restr{\Phi}{ \partial\Omega}=0$.
 \begin{enumerate}[label=(\roman*)]
     \item $\Phi$ and $\eta_\nu$ are smooth on the closure of $\Omega$, hence we can bound the first term. For the second term we have that $\nabla\eta_\nu(x)=\frac{-n(\pi(x))}{\nu}\eta'\left(\frac{\dist(x,\partial\Omega)}{\nu}\right)$, where $n$ is the outward pointing normal on $\partial\Omega$ and $\pi(x)\in\partial\Omega$ is the unique point such that $\dist(x,\partial\Omega)=|x-\pi(x)|$. Thus 
     $$|\nabla\eta_\nu\times\Phi|\leq |\nabla\eta_\nu||\Phi|\leq \frac{C}{\nu}|\eta_\nu'|\nu\leq C.$$
     \item Taking the gradient of $\Phi_\nu=\eta_\nu\curl(\Phi)+\nabla\eta_\nu\times\Phi$, we see that we only have to estimate the terms with the highest power of $\nu$. Thus similarly as in (i) we have $|\nabla^2\eta_\nu(x)|\leq \frac{C}{\nu^2}$, hence
     $$|\nabla^2\eta_\nu\times\Phi|\leq |\nabla^2\eta_\nu||\Phi|\leq \frac{C}{\nu^2}|\eta_\nu'|\nu\leq \frac{C}{\nu}$$
     and similarly for the other terms.
     \item Using the estimate from (i) we have
     $$\norm{\Phi_\nu}_{L^2(\Omega)}^2=\int_{\Omega_\nu}|\Phi_\nu|^2\Diff x\leq C\int_{\Omega_\nu}\Diff x\leq C\nu.$$
     \item Since $\eta_\nu$ does not depend on $t$ and $u\in C^1([0,T];W^{1,2}(\Omega))$ (see Theorem \ref{thm_valli}), the claim follows similarly as for (iii).
     \item Similarly as in (iii) and using the estimate from (ii) we have
$$\norm{\nabla\Phi_\nu}_{L^2(\Omega)}^2=\int_{\Omega_\nu}|\nabla\Phi_\nu|^2\Diff x\leq \frac{C}{\nu^2}\int_{\Omega_\nu}\Diff x\leq \frac{C}{\nu}.$$
 \end{enumerate}

 \subsection{The Hardy Inequality for Boundary Layers}
 \label{appendix_2}
 In the proof of Theorem, \ref{main_th} we used a special type of Hardy inequality that allows to bound the weighted $L^2$-norm of a function on a boundary layer with the $L^2$-norm of its gradient on the same boundary layer. Since a proof of this specific inequality seems hard to find in the literature, we give a proof based on the general theory of Hardy inequalities (see for instance \cite{Kufner}).
\begin{theorem}
Let $f\in H_0^1(U)$ for some open bounded subset $U\subset \bb{R}^n$ with $C^2$ boundary. Further we define the set $U_\epsilon=\{x\in U:\dist(\partial U,x)<\epsilon\}$ for some $\epsilon>0$ small enough. Then
$$\int_{U_\epsilon}\frac{|f|^2}{\dist(x,\partial U)^2}\Diff x\leq C\int_{U_\epsilon}|\nabla f|^2\Diff x,$$
for some constant $C>0$ that does not depend on $\epsilon$.
\end{theorem}

\begin{proof}
First let us define the sets $\Gamma_\epsilon:=\{x\in\bb{R}^n:\dist(x,\partial U)<\epsilon\}$ and $B_\epsilon:=\{x\in U:\dist(x,\partial U)>\epsilon\}$. Note that $\Gamma_\frac{\epsilon}{2}\cup B_\frac{\epsilon}{4}$ forms an open cover of the compact set $\overline{U}$, hence there exists a partition of unity $\{\psi_1,\psi_2\}$ subordinate to this cover such that 
\begin{itemize}
    \item $\psi_1\in C_0^\infty(\Gamma_\frac{\epsilon}{2})$ and $\psi_2\in C_0^\infty(B_\frac{\epsilon}{4})$
    \item $0\leq \psi_i(x)\leq 1$ for all $x\in\Gamma_\frac{\epsilon}{2}\cup B_\frac{\epsilon}{4}$
    \item $\sum_{i=0}^m\psi_i(x)=1$ for all $x\in\overline{U_\epsilon}$.
\end{itemize}
Note that the collection $\{\Gamma_\frac{\epsilon}{2}, B_\frac{\epsilon}{4}\}$ depends on $\epsilon$ hence also the partition of unity $\{\psi_1,\psi_2\}$ depends on $\epsilon$. Hence, considering a standard bump function $\psi\in C_0^\infty(U_\epsilon)$ we see that $\psi_1$ behaves as $\psi(\frac{\cdot}{c\epsilon})$ close to $\partial \Gamma_\frac{\epsilon}{2}$, for some constant $c>0$. Therefore, we have the bounds 
$$\norm{\psi_1}_{L^2(U_\epsilon)}\leq 1\quad\text{and}\quad \norm{\nabla\psi_1}_{L^2(U_\epsilon)}\leq \frac{1}{c\epsilon}.$$
Now let $f\in C_{0}^\infty(U)$ and define $g_i=\psi_i f$.
Hence $g_1\in C_0^\infty(\Gamma_\frac{\epsilon}{2})$ and $g_2\in C_0^\infty(B_\frac{\epsilon}{4})$ and in particular $g_i\in C_0^\infty(U)$. Therefore it follows from the weighted Sobolev embedding for $H_0^1(U)$ that 
$$\left(\int_{U_\epsilon}\frac{|g_1|^2}{\dist(x,\partial U)^2}\Diff x\right)^\frac{1}{2}\leq\left(\int_{U}\frac{|g_1|^2}{\dist(x,\partial U)^2}\Diff x\right)^\frac{1}{2}\leq C_1\norm{g_1}_{H^1(U)}=C_1\norm{g_1}_{H^1(U_\epsilon)}.$$
Here $C_1=2(1+\xi)$ and $\xi>0$ depends only on $\partial U$ (see Theorem 8.4 in \cite{Kufner}).

Since $\dist(B_\frac{\epsilon}{4},\Gamma)>\frac{\epsilon}{4}$ it follows immediately that
$$\int_{U_\epsilon}\frac{|g_2|^2}{\dist(x,\partial U)^2}\Diff x\leq \frac{16}{\epsilon^2}\norm{g_2}_{L^2(U_\epsilon)}^2\leq \frac{16}{\epsilon^2}\norm{f}_{L^2(U_\epsilon)}^2.$$
Hence it follows from $f=g_1+g_2= \psi_1 f+\psi_2 f$ on $U_\epsilon$ and the Minkowski inequality that
\begin{align*}
    \left(\int_{U_\epsilon}\frac{|f|^2}{\dist(x,\Gamma)^2}\Diff x\right)^\frac{1}{2}&\leq \left(\int_{U_\epsilon}\frac{|g_1|^2}{\dist(x,\partial U)^2}\Diff x\right)^\frac{1}{2}+\left(\int_{U_\epsilon}\frac{|g_2|^2}{\dist(x,\partial U)^2}\Diff x\right)^\frac{1}{2}\\
    &\leq C_1\norm{g_1}_{H^1(U_\epsilon)}+\frac{4}{\epsilon}\norm{f}_{L^2(U_\epsilon)}\\
    &\leq C_1\left(\left(\norm{\psi_1}_{L^\infty(U_\epsilon)}+\norm{D\psi_1}_{L^\infty(U_\epsilon)}\right)\norm{f}_{L^2(U_\epsilon)}+\norm{\psi_1}_{L^\infty(U_\epsilon)}\norm{Df}_{L^2(U_\epsilon)}\right)\\
    &\phantom{=}+\frac{4}{\epsilon}\norm{f}_{L^2(U_\epsilon)}\\
    &\leq C_1\left(\left(1+\frac{1}{c\epsilon}\right)\norm{f}_{L^2(U_\epsilon)}+\norm{Df}_{L^2(U_\epsilon)}\right)+\frac{4}{\epsilon}\norm{f}_{L^2(U_\epsilon)}\\
    &\leq\frac{\Tilde{C}}{\epsilon}\norm{f}_{L^2(U_\epsilon)}+\Tilde{C}\norm{Df}_{L^2(U_\epsilon)}\\
    &\leq \Tilde{C}\norm{Df}_{L^2(U_\epsilon)},
\end{align*}
where $\Tilde{C}>0$ only depends on $\partial U$.
For the last step we used Lemma \ref{lemma_p} below. Finally the result follows from the density of $C_{0}^\infty(U)$ in the space $H_{0}^1(U)$. Let $f\in H_0^1(U)$. We take a sequence $(f_n)_{n=1}^\infty\in C_{0}^\infty(U)$ such that $\norm{f_n- f}_{H^1(U)}\to 0$ for $n\to\infty$. Then up to subsequence we can assume that $f_n\to f$ pointwise a.e.~and hence passing to the limit in
$$\int_{U_\epsilon}\frac{|f_n|^2}{\dist(x,\partial U)^2}\Diff x\leq C\int_{U_\epsilon}|\nabla f_n|^2\Diff x,$$ we get

    \begin{align*}
        \int_{U_\epsilon}\frac{|f|^2}{\dist(x,\partial U)^2}\Diff x&=\int_{U_\epsilon}\liminf_{n\to\infty}\frac{|f_n|^2}{\dist(x,\partial U)^2}\Diff x\\
        &\leq \liminf_{n\to\infty}\int_{U_\epsilon}\frac{|f_n|^2}{\dist(x,\partial U)^2}\Diff x\\
        &\leq\liminf_{n\to\infty} C\int_{U_\epsilon}|\nabla f_n|^2\Diff x=C\int_{U_\epsilon}|\nabla f|^2\Diff x.
    \end{align*}

\end{proof}
For the proof above we need a special Poincaré inequality for functions defined on a boundary layer and vanishing on the outer boundary.

\begin{lemma}(\cite{kel3}, Lemma 2.6)
\label{lemma_p}
    Let $\Omega\subset \bb{R}^n$ be a bounded connected open set with $C^2$-boundary $\Gamma:=\partial \Omega$.
    For $\epsilon>0$ let $\Omega_\epsilon=\{x\in \Omega:\dist(x,\Gamma)\leq \epsilon\}$ be the boundary layer of thickness $\epsilon$. Suppose that $u\in H^1_0(\Omega)$, then there exists $\delta>0$ and a constant $C>0$ that only depends on $\Omega,n,\delta$ such that for all $\epsilon < \delta$ 
    $$\norm{u}_{L^2(\Omega_\epsilon)}\leq C\epsilon\norm{\nabla u}_{L^2(\Omega_\epsilon)}.$$
\end{lemma}

\begin{proof}
    Since $\Gamma$ is a $C^2$ and compact manifold of dimension $n-1$, there exists $\delta>0$ such that $\Omega_\delta$ satisfies the unique nearest point property, i.e., for each $p\in\Omega_\delta$ there exists a unique $d(p)\in\Gamma$ such that $\dist(p,\Gamma)=\dist(p,d(p))$. Furthermore, the function $\dist:\Omega_\delta\to\bb{R}$ is $C^2$ (for further details see \cite{foote}). Hence, we can consider a coordinate system $\Tilde{x}=(\Tilde{x}_1,\dots,\Tilde{x}_{n-1})$ on $\Gamma$ and extend it to a coordinate system of $\Omega_\delta$ by defining $(\Tilde{x}(p),\Tilde{x}_n(p))=(\Tilde{x}_1(p),\dots,\Tilde{x}_{n-1}(p),\Tilde{x}_n(p))$ where $\Tilde{x}_n(p)=\dist(p,d(p))$ for every $p\in\Omega_\delta$
    Now let $\epsilon<\delta$. By density of $C_0^\infty(\Omega)$ in $H_0^1(\Omega)$ it is enough to show the inequality for $u\in C_0^\infty(\Omega)$.
    With a change of variables

    \begin{align*}
        \int_{\Omega_\epsilon}|u(x)|^2\Diff x&=\int_\Gamma\int_0^\epsilon|u(\Tilde{x},\Tilde{x}_n)|^2|\det J|\Diff \Tilde{x}_n\Diff \Tilde{x}\\
        &\leq \norm{\det J}_{L^\infty(\Omega_\epsilon)}\int_\Gamma\int_0^\epsilon|u(\Tilde{x},\Tilde{x}_n)|^2\Diff \Tilde{x}_n\Diff \Tilde{x}\\
        &\leq \norm{\det J}_{L^\infty(\Omega_\delta)}\int_\Gamma\int_0^\epsilon|u(\Tilde{x},\Tilde{x}_n)|^2\Diff \Tilde{x}_n\Diff \Tilde{x},\\
    \end{align*}
    where $J=\frac{\partial(x_1,\dots,x_n)}{\partial(\Tilde{x}_1,\dots,\Tilde{x}_n)}$ is the Jacobian matrix.
		
Since $u=0$ on $\Gamma$, we have
\begin{align*}
    |u(\Tilde{x},\Tilde{x}_n)|^2&=|u(\Tilde{x},\Tilde{x}_n)-u(\Tilde{x},0)|^2\\
    &=\left|\int_0^{\Tilde{x}_n}\frac{\partial}{\partial \Tilde{x}_n}u(\Tilde{x},s)\Diff s\right|^2\\
    &\leq \epsilon\int_0^\epsilon\left|\frac{\partial}{\partial \Tilde{x}_n}u(\Tilde{x},s)\right|^2\Diff s.
\end{align*}
Hence we can estimate
\begin{align*}
    \int_{\Omega_\epsilon}|u(x)|^2\Diff x&\leq C\int_\Gamma\int_0^\epsilon|u(\Tilde{x},\Tilde{x}_n)|^2\Diff \Tilde{x}_n\Diff \Tilde{x}\\
    &\leq C\int_\Gamma\int_0^\epsilon\epsilon\int_0^\epsilon\left|\frac{\partial}{\partial \Tilde{x}_n}u(\Tilde{x},s)\right|^2\Diff s\Diff \Tilde{x}_n\Diff \Tilde{x}\\
    &\leq C\epsilon^2\int_\Gamma\int_0^\epsilon \left|\frac{\partial}{\partial \Tilde{x}_n}u(\Tilde{x},\Tilde{x}_n)\right|^2\Diff \Tilde{x}_n\Diff \Tilde{x}\\
     &\leq C\epsilon^2\int_\Gamma\int_0^\epsilon \left|\nabla_{\Tilde{x},\Tilde{x}_n}u(\Tilde{x},\Tilde{x}_n)\right|^2\Diff \Tilde{x}_n\Diff \Tilde{x}\\
     &\leq C\epsilon^2\norm{\det J^{-1}}_{L^\infty(\Omega_\delta)}\norm{J}_{L^\infty(\Omega_\delta)}^2\int_{\Omega_\epsilon} \left|\nabla_{x}u(x)\right|^2\Diff x\\
     &\leq C\epsilon^2\int_{\Omega_\epsilon} \left|\nabla_{x}u(x)\right|^2\Diff x,
\end{align*}
where we used  $\left|\frac{\partial}{\partial\Tilde{x}_n}u(\Tilde{x},\Tilde{x}_n)\right|\leq \left|\nabla_{(\Tilde{x},\Tilde{x}_n)}u(\Tilde{x},\Tilde{x}_n)\right|\leq |J|\left|\nabla_x u(x)\right|.$

\end{proof}

\textbf{Acknowledgements.} The authors would like to thank Stefan \v{S}kondri\'{c} and Alessandro Violini for helpful discussions. E.~W.~is supported by the DFG Priority Programme SPP 2410 (project number 525716336).


\printbibliography 
\nocite{*}

\end{document}